\def\Bbb{\mathbb}

\def\dist{{\rm{dist}}}

\def\dist{{\rm{dist}}}

\def\Bbb{\mathbb}
\def\reals{\Bbb R}
\def\complex{\Bbb C}
\def\disk{\Bbb D}
\def\circle{\Bbb T}

\def\integers{\Bbb Z}
\def\rhp{{\Bbb H}_r}

\def\classS{\cal S}
\def\class2{{\cal S}_2}
\def\class21{{\cal S}^*}
\def\classB{\cal B}

\def\disk{\Bbb D}
\def\cal{\mathcal}

\def\eit{e^{i \theta}}
\def\Julia{{\cal J}}
\def\zbar{{\overline{z}}}

\documentclass[12pt]{amsart}  
\usepackage{amssymb}    

\usepackage{graphicx,amssymb,amsthm,verbatim,amsfonts}

\theoremstyle{plain}                    
\newtheorem{thm}{Theorem}[section]
\newtheorem{cor}[thm]{Corollary}
\newtheorem{lemma}[thm]{Lemma}

\newcounter{ques}
   
\numberwithin{equation}{section}

\begin{document}
\baselineskip=18pt


%

\title [    Models for the Eremenko-Lyubich class ]
          {  Models for the  Eremenko-Lyubich class }


\author {Christopher J. Bishop}
\address{C.J. Bishop\\
         Mathematics Department\\
         SUNY at Stony Brook \\
         Stony Brook, NY 11794-3651}
\email {bishop@math.sunysb.edu}
\thanks{The  author is partially supported by NSF Grant DMS 13-05233.
        }

\date{May 2014}
\maketitle


\begin{abstract}
If $f$ is in the Eremenko-Lyubich class  $\classB$ (transcendental 
entire functions with bounded singular set)  then 
$\Omega= \{ z: |f(z)| > R\}$ and $f|_\Omega$
must satisfy certain  simple topological conditions
when $R$ is sufficiently large. A model $(\Omega, F)$ 
is an open set $\Omega$ and a holomorphic function $F$ on $\Omega$
that satisfy these same conditions. We show that any model 
can be approximated by an 
Eremenko-Lyubich function in  a precise sense.   
In many cases, this allows  the construction of functions in $\classB$ 
with a desired property 
to be reduced to the construction of a model  with that 
property, and this is often much easier to do.
\end{abstract}

\clearpage


\setcounter{page}{1}
\renewcommand{\thepage}{\arabic{page}}
\section{Introduction} \label{Intro} 

The singular set of a entire function $f$ is the closure 
of its critical values and finite asymptotic values and 
is denoted $S(f)$.  The Eremenko-Lyubich class 
$\classB$ consists of functions such that $S(f)$ is a bounded 
set (such functions are also called  bounded 
type).  The Speiser class $\classS \subset \classB$ 
(also called finite type)
are those functions for which $S(f)$ is a finite set.

In  \cite{MR1196102} Eremenko and Lyubich 
showed that if $S(f) \subset \disk_R = \{ z: |z| < R\}$, 
then $\Omega =\{z: |f(z)| >R\}$ 
is  a disjoint  union  of analytic, unbounded, Jordan
domains,  and that $f$ acts a covering map 
$f: \Omega_j \to \{x:|z|>R\}$ on each component $\Omega_j$ 
of $\Omega$.
Building examples where  $\Omega$ has a certain geometry 
 is important for applications to dynamics. We would 
like to start with a model, i.e., a choice of  $\Omega$ and a covering map
$f:\Omega \to \{|z|>1\}$ and ask if $f$ can be approximated 
by an entire function $F$  in  $\classB$ or $\classS$. In  
this paper, we deal with approximation by functions in 
$\classB$.  It turns out that if $\Omega$ satisfies some 
obviously necessary topological conditions, the approximation 
by Eremenko-Lyubich functions is always possible in a 
sense strong enough to imply
that the functions $f$ and $F$ have the same dynamical behavior 
on their Julia sets.  This allows us to build entire functions
in $\classB$  with 
certain behaviors  by simply exhibiting a model with that behavior
(this is often much easier to do).
In \cite{Bishop-S-models} we deal 
with the analogous question for the Speiser class; 
again the approximation is always possible, but in a 
slightly weaker sense (dynamically, given any model we can 
build a Speiser class functions that has the model's 
dynamics on some subset of its Julia set).
In the next few paragraphs we introduce some notation to 
make these remarks more precise.

Suppose $\Omega= \cup_j \Omega_j $ is a disjoint union 
of unbounded simply connected domains such that 
\begin{enumerate}
\item sequences of components of $\Omega$ accumulate only 
     at infinity, 
\item $\partial \Omega_j$ is connected for each $j$
    (as a subset of $\complex$).
\end{enumerate} 
Such an $\Omega$ will be called a model domain. 
If $\overline{\Omega} \cap \{|z|\leq 1\} = \emptyset$, 
we say the model domain is disjoint type.
 The 
connected components $\{ \Omega_j\}$ of $\Omega$ are called tracts. 
Given a model domain,  suppose $\tau : \Omega \to \rhp = 
\{ x+iy: x>0\}$  is holomorphic  so that 
\begin{enumerate}
\item The restriction of $\tau$ to each 
   $\Omega_j$  is a conformal map $\tau_j:\Omega_j \to\rhp$, 
\item If $\{ z_n \} \subset \Omega$ and  
        $\tau(z_n) \to \infty$ then $z_n \to \infty$ .
\end{enumerate} 
Given such a   $\tau:\Omega \to \rhp$, we 
call $F(z) = \exp(\tau(z))$ a model function.

The second condition  on $\tau$ 
is a careful way of saying that the 
conformal map on each component sends $\infty$ to 
$\infty$. Even after making this condition, we still 
have a  (real) 2-dimensional family of conformal maps 
from each component of $\Omega$ to $\rhp$ determined by 
choosing where one base point in each component will map 
in $\rhp$. 
 A choice of both a model domain $\Omega$
and a model function $F$ on $\Omega$ will be called 
a model.

Given a model $(\Omega, F)$  we let 
$$ \Omega(\rho) = \{ z \in \Omega : |F(z)| > e^\rho\} = 
   \tau^{-1}(\{ x+iy: x > \rho\}),$$
and 
$$ \Omega(\delta, \rho) = \{ z \in  \Omega :  e^\delta < |F(z)| <  e^\rho\} = 
   \tau^{-1}(\{ x+iy: \delta <  x  < \rho\}). 
$$
If $\Omega $ has connected components $\{ \Omega_j\}$  we 
let $\Omega_j(\rho) = \Omega(\rho) \cap \Omega_j 
$ and similarly for $\Omega_j(\delta, \rho)$. 

\begin{figure}[htb]
\centerline{
	\includegraphics[height=3.5in]{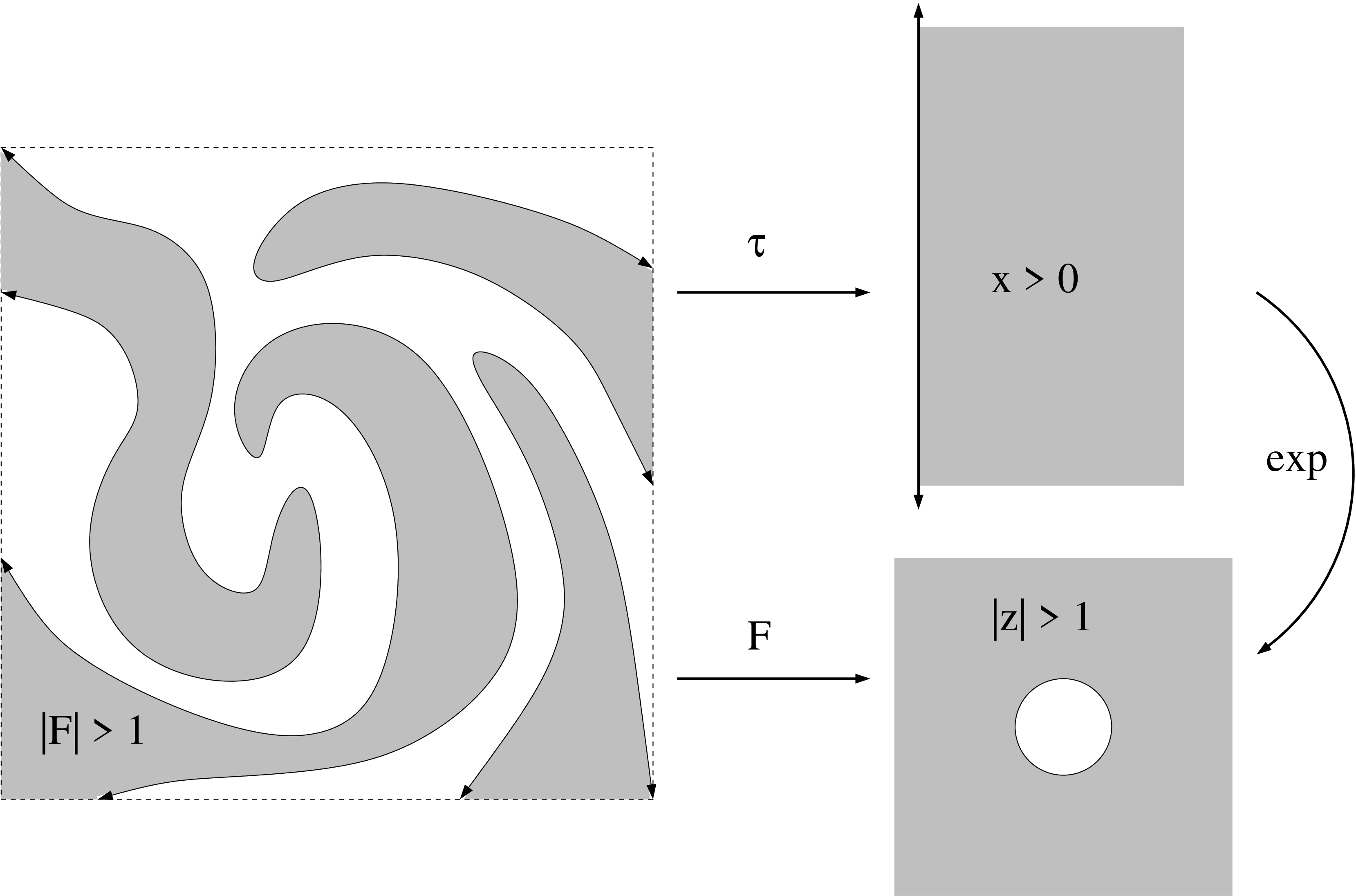}
}
\caption{ \label{LevelSet1}
A model consists of an open set $\Omega$ which may 
have a number of connected components called tracts. 
Each  tract is mapped conformally by $\tau$ to the 
right half-plane and then by the exponential function 
to the exterior of the unit disk.  The composition
of these two maps is the model function $F$. 
In this paper, we are interested in knowing 
if  a holomorphic model function on 
$\Omega$  can be approximated 
by holomorphic function on the entire plane.
}
\end{figure}

Moreover, a model has dynamics:
we can iterate $F$ as long as the iterates keep landing 
in $\Omega$, and we define the Julia set of a model  
$$ \Julia(F) = \bigcap_{n\geq 0} \{ z\in \Omega: F^n(z) 
         \in \Omega \}.$$

Each function $f$ in the Eremenko-Lyubich  class that 
satisfies $S(f) \subset \disk$  gives rise 
to a model  by taking $\Omega = \{z : |f(z)|> 1\}$ and 
 $\tau(z) = \log f(z)$.
The $\log$ is well defined since each component of $\Omega$ is 
simply connected and $f$ is non-vanishing on $\Omega$. 
Eremenko and Lyubich proved in \cite{MR1196102}  
that $\tau$ defined in this way is a conformal map from 
each  component of $\Omega$ to $ \rhp$.
We call a model arising in this way an Eremenko-Lyubich model. 
If $f$ is in the Speiser class, we call it a Speiser model.

An entire function $f$ is called 
hyperbolic if $f \in \classB$ and if 
there is a compact  
set $K$ so that $f(K) \subset \rm{int}(K)$ and 
$f: f^{-1}(\complex \setminus K) \to \complex \setminus K$
is a covering map. This is equivalent to saying that the 
singular set is bounded and every point of $S(f)$ iterates
to an attracting periodic cycle in the Fatou set.
 If we can take $K$ to be connected, 
then $f$ is called disjoint type. 
This implies the Fatou set of $f$ (i.e., the set where the iterates
of $f$ form a normal family) is connected (e.g., see \cite{Rempe-App}).
The assumption that $S(f) \subset \disk$ and 
$\Omega = \{ z:|f(z)| >1\}$  implies that $f$ is 
disjoint type. 
In this case, the usual Julia set of $f$ (defined 
as the complement of the Fatou set)    is the 
same set as the Julia set of the corresponding 
model 
$$ \Julia(f) = \bigcap_{n \geq 0} \{ z:  |f^n(z)|  \geq  1\}.$$
Thus we can refer to $\Julia(f)  $  where 
we think of $f$ as either an entire function in 
$\classB$ or as a model function on $\Omega =\{ x:|f(z)|> 1\}$
 without ambiguity.
Basic facts about hyperbolic and disjoint type functions 
are discussed in \cite{BRW14}.

The question now arises of whether  or not
the Eremenko-Lyubich models are 
only a very special subclass of  general models.
There are at least two ways to make such a 
comparison: geometric and dynamical.  We start 
with our geometric result.

A homeomorphism of the plane is called quasiconformal 
if it is absolutely continuous on almost all vertical 
and horizontal lines and the partial derivatives 
$f_z = f_x -i f_y$ and $f_\zbar = f_x + i f_y$
almost everywhere  satisfy 
$$  |f_\zbar| \leq k |f_z|,$$
where $0 \leq k < 1$. Geometrically, the derivative 
of $f$ exists almost everywhere and sends circle
to ellipses of eccentricity at most $K = (1+k)/(1-k)$. 
This number $K$ is called the quasiconstant of $f$.
The ratio $\mu = f_\zbar / f_z$ is called the complex
dilatation of $f$. The measurable Riemann mapping 
theorem (see e.g., \cite{MR2241787}, \cite{MR0344463}) 
says that given any measurable $\mu$ with $|\mu|< k < 1$, 
there is a quasiconformal homeomorphism $\varphi$ of the 
plane so that the complex dilatation of $\varphi$ equal $\mu$ 
almost everywhere. We shall actually use the following 
consequence of this: if $\psi: \Omega \to \Omega'$ is a 
a  quasiconformal map between planar domains, then there
is a quasiconformal map $\varphi: \complex \to \complex $
so that $\psi \circ \varphi$ is conformal on $\varphi^{-1}(\Omega)$.

We can now state our main result:

\begin{thm}[All models occur] \label{QR}
Suppose $(\Omega, F)$ is a model and $0 < \rho  \leq 1$.
Then there is $f \in \classB$ and a quasiconformal $\varphi: \complex 
\to \complex$ so that  $ F = f \circ \varphi$ on $  \Omega(2\rho).$
In addition, 
\begin{enumerate}
\item $|f \circ \varphi| \leq  e^{2\rho}$ off $\Omega(2\rho$)
     and  
    $|f \circ \varphi| \leq  e^{\rho}$ off $\Omega(\rho$).
Thus the components of $\{z: |f(z)| > e^\rho\}$ are 
     in $1$-to-$1$ correspondence to the components of $\Omega$
     via  $\varphi$.
\item $S(f) \subset D(0, e^\rho)$.
\item 
the quasiconstant of
     $\varphi$ is $ O(\rho^{-2})$
     with a constant  independent of  $F$ and $\Omega$,
\item  $\varphi^{-1}$ is conformal except on the set 
          $\Omega(\frac \rho 2,2 \rho).$
\end{enumerate}
\end{thm} 

Another useful way to state the result (for those 
familiar with the language), is that for any model 
$F$ and any $\rho >0$, $F$ restricted to $\Omega(\rho)$ 
can be extended to a quasiregular function on 
$\complex$ that is bounded off $\Omega(\rho)$  and 
has   a quasiconstant bounded  depending only 
on $\rho$. The extension is holomorphic off 
$\Omega(\rho/2)$.  The precise definition and basic 
properties of quasiregular functions can be 
found in, e.g., 
 \cite{MR1859913}, 
\cite{MR0344463},
 \cite{MR994644},
\cite{MR1238941}.

We say that two  continuous 
 maps   $f:X \to X$ and $g: Y\to Y$ are 
conjugate if there is a homeomorphism 
$h:X \to Y$ so that 
$$ g = h \circ f \circ h^{-1}.$$
It is easy to see that if this holds then  
$$ g^n  = h \circ f^n \circ h^{-1}, $$
for all $n \geq 0$, so that the orbits of $f$ correspond 
via $h $ to the orbits of $g$. For our purposes
this means the dynamics of $f$ and $g$ are ``the same''.

Lasse Rempe-Gillen has pointed out that Theorem 
\ref{QR} implies the following result: 

\begin{thm} \label{EL models}
If $F$ is any disjoint type  model, then  there is a disjoint type
Eremenko-Lyubich 
function $f$ so that $f$ and $F$ are quasiconformally
conjugate on a neighborhood of their Julia sets.  
More precisely, there is a quasiconformal $\varphi: \complex
\to \complex$ so that 
$$ f \circ \varphi = \varphi \circ F,$$
on  an open set that contains both $\Julia(f)$ and $\Julia(F)$.
\end{thm} 

 This means that any property 
of $\Julia(F)$ that is preserved by quasiconformal maps 
also holds for $\Julia(f)$, e.g., every component 
of $\Julia(f)$ is path connected or the Julia set 
has positive area. Since it is generally 
easier to build a  model with a desired property than 
to build a entire function directly, this result is 
useful in constructing Eremenko-Lyubich functions
with pathological behavior. For example, 
Rempe-Gillen uses this result in  \cite{Rempe-arc-like}
 to show there are 
functions in $\classB$ so that the components of
the Julia sets are pseudo-arcs,  by building 
models that have this property.

 Theorems \ref{QR} and \ref{EL models} are   inspired 
by results of Lasse Rempe-Gillen
\cite{Rempe-App} that draw the same conclusions
from a stronger hypothesis: he assumes that $F = e^\tau$
is defined on a model domain $\Omega$ with a single 
tract  and 
restricts it to a slightly smaller domain than $\Omega(\rho)$; 
 roughly, he omits a strip whose width grows
logarithmically, 
i.e.,  $\tau^{-1}(\{x+i y: x > \max(1, \log |y|\})$.
His version of Theorem \ref{QR} is proved by using a 
Cauchy integral construction to first approximate $F$ uniformly 
and then show that uniform approximation implies quasiconformal 
approximation in the sense of Theorem \ref{QR}. 
Rempe-Gillen then shows how to  deduce Theorem \ref{EL models} from 
Theorem \ref{QR} using an iterative construction. With his 
permission, we sketch his argument  in Section \ref{Rempe rigidity}
for the convenience of the reader (our application does not 
require the much more powerful results  he also proved in 
\cite{MR2570071}).

We sketch the proof of Theorem \ref{QR} quickly here to 
give the basic idea.
 Let  $W = \complex \setminus \overline{\Omega(\rho)}$. 
It is  simply connected, non-empty and not the whole plane,
 so there is a conformal map $\Psi: W \to \disk$. 
Since $\Psi $ maps  $\partial W $ 
to the unit circle, if we knew that $F=f|_\Omega$ for
some entire  function $f$, then  $ B= e^{-\rho} \cdot F \circ \Psi^{-1}$  
would be an inner function on $\disk$ (i.e., a  holomorphic function 
on $\disk$ so that $|B|=1$ almost everywhere on the boundary).

The  proof of Theorem \ref{QR}  reverses this 
observation. Given the model and the corresponding domain 
$W$ and conformal map $\Psi$ 
we construct a Blaschke product $B$ (a special type of inner function)
 on the disk so that 
$G=B \circ \Psi $ approximates $F=e^\tau$  
on $\partial \Omega(\rho)$
(the precise  nature of the approximation will be described later).
This  step is fairly straightforward using standard estimates of 
the Poisson kernel on the disk. We then ``glue'' $G  $ to $F$
 across $\partial W$ to get a quasi-regular 
function $g$ that agrees with $F$ on $\Omega(2\rho)$ and agrees 
with $G$ on $W$.
This takes  several (individually easy) steps to accomplish.
We then use the measurable 
Riemann mapping theorem to define a quasiconformal mapping $\phi: 
\complex \to \complex$ so that $f = g \circ \phi$ is holomorphic
on the whole plane.
The only critical 
points of $g$ correspond to critical points of $B$, and 
critical points introduced into $\Omega(\rho, 2\rho)$
by the gluing process.  We will show that  both types
of   critical values have absolute 
value $ \leq e^\rho$. A different argument shows that any 
finite asymptotic value of $f$ must correspond to a limit 
of $B$ along a curve in $\disk$, so  all finite asymptotic
values of $f$ are also bounded by $e^\rho$. Thus $f \in \classB$.
Since $g$ is only non-holomorphic in $\Omega( \rho, 2 \rho)$, 
we will also get that $\phi^{-1}$ is conformal everywhere except 
in $\Omega(\rho, 2 \rho)$.

Given Theorem \ref{QR} for the Eremenko-Lyubich class $\classB$,
 it is natural to ask  the analogous question for the more restrictive 
Speiser class: can every model be approximated by a Speiser model? 
This question is addressed in \cite{Bishop-S-models}, where an 
analog of Theorem \ref{QR} is proven for the Speiser class. 
In that paper we show that given a model  $(\Omega, F)$
 and  any $\rho >0$,  there is a $f \in \classS$ 
and a quasiconformal map $\phi : \complex \to \complex $  so that 
$ f \circ \phi = e^{\tau}$   on  $\Omega(2\rho)$.
We may take $\phi$  to be   conformal on $\Omega(\rho )$, 
and $f$ may be taken with the two critical values $\pm e^\rho$
and no finite asymptotic values.

Note that this result omits the  conclusion
 ``$ f \circ \phi$ is bounded off $\Omega(2 \rho)$''.
In fact,  the Speiser functions constructed in 
\cite{Bishop-S-models} will usually 
be unbounded off $\Omega$;  $f$ can 
have  ``extra'' tracts that do not correspond to 
tracts of the original model function $F$. 
It is shown in \cite{Bishop-S-models} that $f$ has 
at most twice as many tracts as $F$ and sometimes
this many are needed.
The Speiser version of 
Theorem \ref{EL models} says that  
if $(\Omega, F)$ is any model, there is a Speiser class
function $f$, a closed set $A \subset \Julia(f)$,
an open neighborhood $U$ of $A$
and a quasiconformal map $\varphi: \complex \to 
\complex$ that conjugates $f$ to $F$ on $U$.
Thus the dynamics of any model can be found ``inside'' 
the Julia set of a Speiser class function. 
See \cite{Bishop-S-models} for the precise statement.

Finally, the construction in this paper uses a construction 
called ``simple folding''. A more complicated version of this 
is used in \cite{Bishop-classS} to construct functions in the 
Speiser class with prescribed geometry. The paper \cite{Bishop-S-models}
on Speiser models uses the main result of \cite{Bishop-classS}
to prove the results described in the preceding paragraphs. 
Thus this paper can be thought of as a  gentle introduction 
to  \cite{Bishop-classS}, whereas \cite{Bishop-S-models} is a sequel 
to \cite{Bishop-classS}. 
The results of both this paper and \cite{Bishop-S-models}
originally appeared in a single  preprint titled ``The geometry of bounded
type entire functions'', but I have split this into two papers 
in an attempt to improve the exposition and to separate the simpler, 
self-contained arguments   for the Eremenko-Lyubich class $\classB$ from
 the more intricate arguments relying on \cite{Bishop-classS} needed
for the Speiser class $\classS$.

Using quasiconformal methods to construct 
holomorphic functions with desired geometry 
has a long history and  
has been a crucial tool in several areas such 
as value distribution theory and, more
recently, holomorphic dynamics. See 
\cite{MR2121873} and \cite{MR2435270} for surveys 
of  applications to the first area and 
\cite{BrannerBook} for a recent survey of the 
second.

Many thanks to Adam Epstein, Alex Eremenko and Lasse Rempe-Gillen
for very helpful discussions about the contents of this paper and 
about the folding constructions and their applications. 
The  introduction of the paper and the formulation of 
the main result in terms of approximation of models 
was inspired by a lecture of Lasse Rempe-Gillen at
an ICMS conference on transcendental dynamics in Edinburgh, May 2013.
Also thanks to the anonymous referee whose comments prompted 
the revision of the entire manuscript and improved 
its clarity and correctness. 
Malik Younsi also  read the revised manuscript and 
gave me numerous corrections and suggestions that 
I greatly appreciate.

\section{Reduction of  Theorem \ref{QR} to the case $\rho=1$} \label{reduction}

We start the proof of Theorem \ref{QR} with the observation that 
it suffices to prove the result for $\rho=1$.

To do this we define two quasiconformal maps, $\psi_\rho$
and $\varphi_\rho$.  
Define 
$$
L(x)  =
\begin{cases}
   x,                 & 0 < x < \rho/2, \\
   (\frac {2-\rho}{\rho})(x-\rho/2)+\rho/2     &    \rho/2 \leq x \leq \rho,\\
   x /\rho  &    \rho \leq x \leq 2\rho.\\
\end{cases}
$$
This is a piecewise linear  map
that sends  $[\rho/2,\rho]$ to $[\rho/2,1]$ and sends 
$[\rho, 2 \rho]$ to $[1 ,2]$. The slope on both intervals 
is less than $2 /\rho$.
For $z = x+ i y \in \rhp$,  define
$$
\sigma_\rho (z) =
\begin{cases}
    L(x) + iy   &     0  < x \leq 2\rho,\\

    z+2-2\rho &   x > 2\rho.  \\
\end{cases}
$$
This is quasiconformal $\rhp \to \rhp$ with quasiconstant 
$K  \leq  2/\rho$.  Then set 
$$
\psi_\rho (z) =
\begin{cases}
    z, &  z \not \in \Omega  \\
    \tau_j^{-1} \circ \sigma_\rho \circ \tau_j(z),  &
  z \in \Omega_j .\\
\end{cases}
$$
Note that $\psi_\rho$ is the identity near 
$\partial \Omega$, so  $\psi_\rho$ is quasiconformal 
on the whole plane by the 
 Royden gluing lemma, e.g., 
Lemma 2 of \cite{MR0422691}, Lemma I.2 of \cite{MR816367} 
on page 303, or  \cite{MR0254234}. 
(Actually, since $\psi_\rho$ is the identity off 
$\Omega(\rho/2)$ which has a smooth boundary, one can 
use a weaker version of the gluing lemma.)

 Next, define 
$$
\varphi_\rho (z) =
\begin{cases}
    z, & |z| < e^{\rho/2}  \\
     \exp(\sigma_\rho(\log(z))), &  |z| \geq e^{\rho/2}  \\
\end{cases}.
$$
Note that even though $\log(z)$ is multi-valued, 
the function $\sigma_\rho$ does not change  the imaginary part 
of its argument, so the exponential of $\sigma_\rho(\log(z))$ is 
well defined.  This is clearly a quasiconformal map of the plane 
with quasiconstant $2/\rho$.
Note also  that these functions were chosen so that if 
$F = \exp \circ \tau$ is the model function associated 
to $\Omega $ and $\tau$, then  on $\Omega_j$ 
 \begin{eqnarray} \label{F conj}
 F \circ \psi_\rho 
\nonumber
&=&\exp \circ \, \tau_j \circ \tau_j^{-1} \circ \sigma_\rho \circ \tau_j \\
&=& \exp \circ \,  \sigma_\rho \circ \log \circ \exp \circ \tau_j\\
\nonumber
&=& \varphi_\rho \circ F.
\end{eqnarray}

Now apply Theorem \ref{QR} to the model $(\Omega, F)$ 
with $\rho=1$ to get a $f \in \classB$  and
a quasiconformal map $\Phi: \complex \to 
\complex$ so that  
$f \circ \Phi = F$ on $\Omega(2)$ and $S(f) \subset D(0, e^1)$. 
 Let $g_\rho = \varphi^{-1}_\rho \circ f \circ \Phi  \circ \psi_\rho$. 
This is an entire function pre and post-composed with 
quasiconformal maps of the plane, so it is quasiregular.
By the measurable Riemann mapping theorem, there
is a quasiconformal $\Phi_\rho: \complex \to \complex$ 
so that $f_\rho = g_\rho \circ \Phi^{-1}_\rho$ is entire
and clearly 
$$S(f_\rho) = S(g_\rho)  \subset \varphi_\rho^{-1} (S(f))
   \subset \varphi_\rho^{-1}(D(0,e)) =  D(0,e^\rho).$$
For $z \in \Omega(2\rho)$, $\psi_\rho(z) \in \Omega(2)$, so
using this and (\ref{F conj}) 
 \begin{eqnarray*} 
f_\rho \circ \Phi_\rho(z) 
&=& g_\rho(z)  \\
&=& \varphi^{-1}_\rho  ( f(\Phi(  ( \psi_\rho(z)))) \\
&=& \varphi^{-1}_\rho  ( F  ( \psi_\rho(z))) \\
&=& F(z).
\end{eqnarray*}
Similarly, $|f_\rho \circ \Phi_\rho|=|g_\rho|$ is 
bounded by $e^{2\rho}$ off $\Omega(2 \rho)$. 
The quasiconstant of $\Phi_\rho$ is, at worst, the 
product of the constants for $\Phi$, $\psi_\rho$ and $\varphi_\rho$, 
which is $K_1 \cdot 4  \rho^{-2}$, where $K_1$ is the 
upper bound for the quasiconstant in 
Theorem \ref{QR} in the case $\rho =1$.  

Finally, our construction in the next section 
will show that $\Phi$ is conformal except on 
$\Omega(1,2)$ and that $F$ has 
a quasiregular extension to the plane that is 
holomorphic except on $\Omega(1,2)$ and is bounded 
by $e$ off $\Omega(1)$ and by $e^2$ off $\Omega(2)$.
This implies 
that   $g_\rho$ is holomorphic except on  $\Omega(\rho/2,2 \rho)$ 
(since $\psi_\rho$ is holomorphic off $\Omega(\rho/2,2\rho)$
and $\varphi_\rho^{-1} $ is holomorphic off
 $\{e^{\rho/2} < |z| < e^2 \}$.)
 This, in turn, implies 
that $\Phi_\rho$ is conformal except on 
$\Omega(\rho/2, 2\rho)$, as desired. 
Thus $f_\rho$ satisfies Theorem \ref{QR} for the model 
$(\Omega, F)$ and the given $\rho > 0$.

\section{The proof of Theorem \ref{QR}} \label{proof of thm QR}

In this section we give the  proof of Theorem \ref{QR}
for $\rho=1$, 
stating certain facts as lemmas to be proven in later 
sections.

Let $W =\complex \setminus  \overline{\Omega(1)}$. 
This is an open, connected, simply connected domain that
is bounded by analytic arcs $\{\gamma_j\}$ that are each 
unbounded in both directions. See Figure \ref{WandOmega}.
The same comments hold for the larger domain 
 $W_2 =\complex \setminus  \overline{\Omega(2)}$.

\begin{figure}[htb]
\centerline{
\includegraphics[height=2.2in]{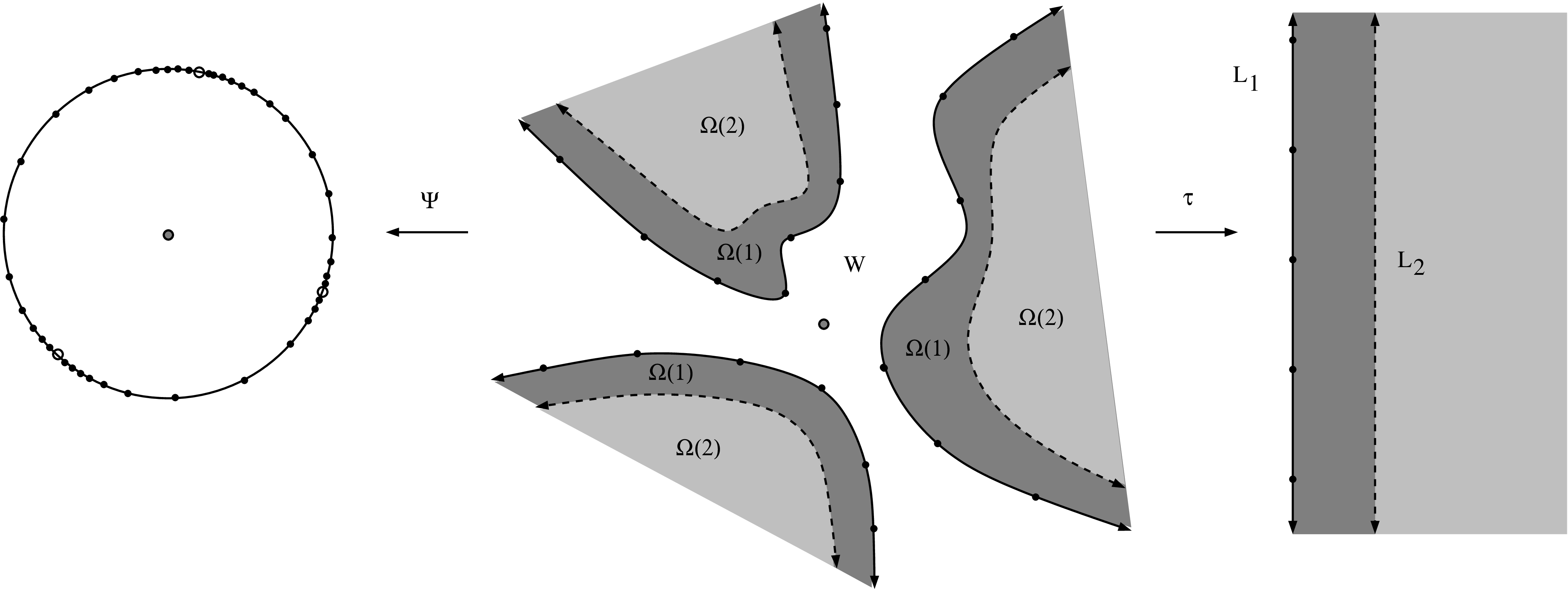}
}
\caption{ \label{WandOmega}
$W$ is the complement of $\Omega(1)$; it is simply connected and 
bounded by smooth curves. We are given the holomorphic 
 function $F = e^\tau$ on 
$\Omega(2)$ and we will define a holomorphic  function on $W$
using the Riemann map $\Psi$ of $W$ to the unit disk, and a specially 
chosen infinite Blaschke product $B$ on the disk. We will then 
interpolate these functions in $\Omega(2)\setminus \Omega(1)$
by a quasiregular function. 
Each component of this set is mapped to a vertical strip by 
$\tau$, and it is in these strips that we construct the 
interpolating functions. Note that the integer partition 
on the boundary of the half-plane pulls back under $\tau$ 
to a partition of each component of $\partial \Omega(1)$, and  
that $\Psi$ maps these to a partition of the unit circle 
(minus the singular set of $\Psi$). The Blaschke product $B$
will be constructed  so that $B^{-1}(1)$ approximates 
this partition of the circle.
}
\end{figure}

Let $L_1 =\{x+iy: x=1\}$ and 
$L_2 =\{x+iy: x=2\}$. The vertical strip between these 
two lines will be denoted $S$.   Note that 
$L_1$ is partitioned into intervals of length 
$2 \pi$ by the points $ 1+2 \pi i \integers$. 
This partition of $L_1$ will be denoted ${\cal J}$.
Note that $\tau_j(\gamma_j) = L_1$, so each 
curve $\gamma_j$ is partitioned by the 
image of ${\cal J}$ under $\tau_j^{-1}$.
We denote  this partition of $\gamma_j$ by ${\cal J}_j$.
Because elements of ${\cal J}_j$ are all images 
of a fixed interval $J \in L_1 \subset \rhp$ under
some conformal map of $\rhp$,  the distortion 
theorem (e.g., Theorem I.4.5 of \cite{MR2150803}) implies 
they  all lie in a compact family of smooth arcs and 
that adjacent elements of ${\cal J}_j$ have comparable 
lengths with a uniform constant, independent of $j$, 
$\Omega$ and $F$.

Let $\Psi: W \to \disk$ be a conformal map given by the 
Riemann mapping theorem.
We claim  that $\Psi$ can be analytically continued from $W$ to 
$W_2$  across $\gamma_j$.   
Let $R_1$ denote reflection across  $L_1$ 
and for $ z \in \Omega_j \cap W = \tau_j^{-1}
(\{ x+iy: 0 < x< 1\})$  let  $T = \tau_j^{-1} \circ R_1
\circ \tau_j$; this defines an anti-holomorphic $1$-to-$1$ map 
from $\Omega_j(0,1)$ to $\Omega_j(1,2)$ that fixes 
each point of $\gamma_j$. We can then 
extend $\Psi$ by the formula
 $$ \Psi(T(z)) = 1/\overline{\Psi(z)},$$   
(where the right hand side denotes reflection of $\Psi(z)$ across 
the unit circle). 
The Schwarz reflection principle says this is an 
analytic continuation of $\Psi$ to $W_2$.

Thus  $\Psi$ is a smooth map of each $\gamma_j$ onto 
an arc $I_j$ of the unit circle $\circle =  \partial \disk 
= \{ |z| =1\}$. The complement of these
arcs is a closed set $E \subset \circle$.
 It is a standard 
fact of conformal mappings that  since $E$   is the set where
a conformal map fails to have a finite limit, it has 
zero Lebesgue, indeed, zero logarithmic capacity. We 
will not need this fact, although we will use the easier 
fact that $E$ can't contain an interval (i.e., a conformal 
map can't have infinite limits on an interval). 

 The partition ${\cal J}_j$ of $\gamma_j$ transfers, 
via $\Psi$ to a   partition of $I_j \subset \circle$ into 
infinitely many intervals $\{ J_k^j\}, k \in \integers$.
We will let ${\cal K} = \cup_{j,k} J_k^j $ denote the  collection 
of all intervals that occur this way. Thus $\circle 
= E \bigcup \cup_{K \in{\cal K}} K$.

Because $\Psi$  conformally extends   from $W$ to   $W_2$,
  $|\Psi'|$ has comparable minimum 
and maximum on each partition element of $\gamma_j$ (with 
uniform constants). Thus the corresponding intervals $\{ J^j_k\}$   
have the property that adjacent intervals have comparable 
lengths (again with a uniform bound).

The hyperbolic distance between two points $z_1, z_2 \in \disk$
is defined as 
$$ \rho(z_1,z_2) = \inf_\gamma \int_\gamma \frac {|dz|}{ 1-|z|^2} .$$
See Chapter 1 of \cite{MR2150803} for the basic properties of 
the hyperbolic metric. Here we will mostly need the facts
that it is invariant under M{\"o}bius  self-maps of the 
disk, that hyperbolic geodesics are circular arcs in 
$\disk$ that are perpendicular to $\circle$, and that
 points  hyperbolic distance $r$ from $0$ are 
Euclidean distance 
$$ \frac 2{\exp(2r)+1} = O(\exp(-r)),$$
from the unit circle.

For any proper sub-interval $I \subset \circle$, let
$\gamma_I$ be the hyperbolic geodesic with the same
endpoints as $I$ and let $a_I$ be the point on 
$\gamma_i$ that is closest to the origin (closest
in either the Euclidean or hyperbolic metrics; it is 
the same point).

Since $ {\cal K}$ are disjoint intervals on the circle, 
$$ \sum_{K \in {\cal K}} (1-|a_K|) < \infty,$$
and so   
$$ B(z) = \prod_{\cal K} \frac {|a_K|}{a_K}
 \frac  {  a_K-z}{ 1- \overline{a_K} z}, $$
 defines a convergent  Blaschke product (see Theorem II.2.2 of 
\cite{MR628971}). 
Thus $B$ is a bounded, non-constant, holomorphic function 
on $\disk$ that vanishes exactly on the set $\{ a_n\}$. 
Also, $|B|$ has radial limits $1$ almost 
everywhere.  Moreover, $B$ extends meromorphically to 
$\complex \setminus E$, where $E$ is the accumulation set 
of its zeros on $\circle$; this is the same set $E$ as 
defined above using the map $\Psi$ (the zeros accumulate 
at both endpoints of every component of $\circle 
\setminus E$, and since these points are dense in 
$E$, the accumulation set of the zeros is the whole 
singular set $E$).
The poles of the extension are precisely  the 
points in the exterior of the unit disk that are 
the reflections  across $\circle$ of the zeros. 

Any subset ${\cal M}$ of ${\cal K}$ also  defines a 
convergent Blaschke product. Fix such a subset.
The corresponding
 Blaschke product $B_{\cal M}$
 induces a partition of each $I_j$ with 
endpoints given by the set $\{\eit: B_{\cal M}(\eit) =1 \}$ and this 
induces a partition ${\cal H}_j$ of  each $\gamma_j$ via the 
map $\Psi$. This in turn, induces a partition ${\cal L}_j$ of
$L_1$ via $\tau_j$.

 We would like to say that the partitions 
${\cal L}_j$ and ${\cal J}$ are ``almost the same''. The 
first step to making this precise is a lemma that
we will prove in Section \ref{Blaschke sec}:  

\begin{lemma} \label{Blaschke lemma 1}
There is a subset ${\cal M} \subset {\cal K}$ so
that if $B$ is the Blaschke product corresponding to
${\cal M}$ and ${\cal L}_j$ is the partition of $L_1$
corresponding to $B$ via $\tau_j \circ \Psi^{-1}$, then
each element of ${\cal J}$ hits at least $2$
elements of ${\cal L}_j$ and at most $M$
 elements of ${\cal L}_j$, where $M$ is uniform.
 In particular, no element of ${\cal J}$ can
hit both endpoints of any element of ${\cal L}_j$ (elements
of each partition are considered
as closed intervals).
\end{lemma}

In Section \ref{linearize sec} we will prove

\begin{lemma} \label{linearize}
Suppose $K =  [1+ia,1+ib] \in  {\cal L}_j$ and define
$$ \alpha(1+iy) = \frac 1{2\pi} \arg( B \circ \Psi \circ \tau_j^{-1}(1+iy)),$$
where we choose a branch of $\alpha$  so $\alpha(1+ia) =0$
(recall that $B(\Psi(\tau_j^{-1}(1+ia))) =1 \in \reals$). Set
$$\psi_1(z) =   1+i(a(1-\alpha(z)) + b\alpha(z)) =1+ i(a + (b-a)\alpha(z)).$$
Then $\psi_1$ is a homeomorphism from $K$ to itself so  that
$ \alpha \circ \psi^{-1}_1 :K \to [0, 1]$ is linear
and $\psi_1$ can be extended to a quasiconformal
homeomorphism of $R= K  \times [1,2]$ to itself
that is the identity on the $\partial R \setminus K$
(i.e., it fixes points on the top, bottom and right side
of $R$).
\end{lemma}

The main point  of the proof is  to show that 
 $\arg(B\circ \Psi \circ \tau_j^{-1}) : K \to [0, 2 \pi]$ is biLipschitz 
with uniform bounds.

Roughly, Lemma \ref{Blaschke lemma 1} 
says  there are more elements of ${\cal J}$ than 
there are of ${\cal L}_j$.  This is made a little 
more precise by the following:

\begin{lemma} \label{partition map}
There is a $1$-to-$1$, order preserving map of ${\cal L}_j$ into 
(but not necessarily onto) ${\cal J}$ so that 
each interval $K \in {\cal L}_j$ is sent to 
an interval $J$ with $\dist(K, J) \leq 2 \pi$. 
Moreover, adjacent elements of ${\cal L}_j$ map 
to elements of ${\cal J}$ that are either adjacent
or are separated by an even number of elements 
of ${\cal J}$.  
\end{lemma}

This will be proven in Section \ref{partition sec}.  Again, 
the proof is quite elementary.

Partition ${\cal J} = {\cal J}_1^j \cup {\cal J}_2^j$
according to whether the interval is  associated 
to some element of ${\cal L}_j$ by  Lemma \ref{partition map}
(i.e., ${\cal J}_1^j$ is the image of ${\cal L}_j$ under
the map in the lemma). The maximal chains of 
adjacent elements of ${\cal J}_2^j$ will be called 
blocks. By the lemma, each block has an even 
number of elements. We will say that the block 
associated to an element $J \in {\cal J}_1^j$ is 
the block immediately above $J$. 

Thus each  interval $K$  in  ${\cal L}_j$ is associated to 
 an interval   $J'$  that  consists of the corresponding 
$J $ given by Lemma \ref{partition map} and its associated block.
$K$ and $J'$ have comparable lengths and are close to 
each other, so the  orientation preserving linear 
map  from $J'$ to $K$ defines a piecewise linear 
map $\tilde \psi_2: \reals   \to \reals$ that is biLipschitz with a
uniform constant. Using linear interpolation we can 
extend this to a biLipschitz map $\psi_2$ of the 
strip $S=\{ x+iy: 1 < x < 2\}$ to itself 
 that equals $ \tilde \psi_2$ on
$L_1$ (the left boundary) and is the identity on
$L_2$ (the right side).

Each element 
$J \in {\cal J}_2^j$ is paired with a distinct element
$ J^* \in {\cal J}_2^j$ that belongs to the same block.
The two outer-most elements of the block are paired, 
as are the pair adjacent to these, and so on. 
Similarly, each point $z$ is paired with the 
other  point $z^*$ in the block that has the 
same distance to the boundary (the center of the 
block is an endpoint of ${\cal J}$ and is paired 
with itself).

For  each $K \in {\cal L}_j$, let $J_K $  be the corresponding 
element of ${\cal J}_1^j$ and let $I_K$ be the union of 
$J_K$ and its corresponding block.  Let $R_K = [1,2] 
\times I_K$.  Let $U_K = R_K \setminus X_K$, where 
$X_K$ is the  closed segment    
connecting the upper left corner of $R_K$ to the 
center of $R_K$. See Figure \ref{Uk}. 
\begin{figure}[htb]
\centerline{
\includegraphics[height=2in]{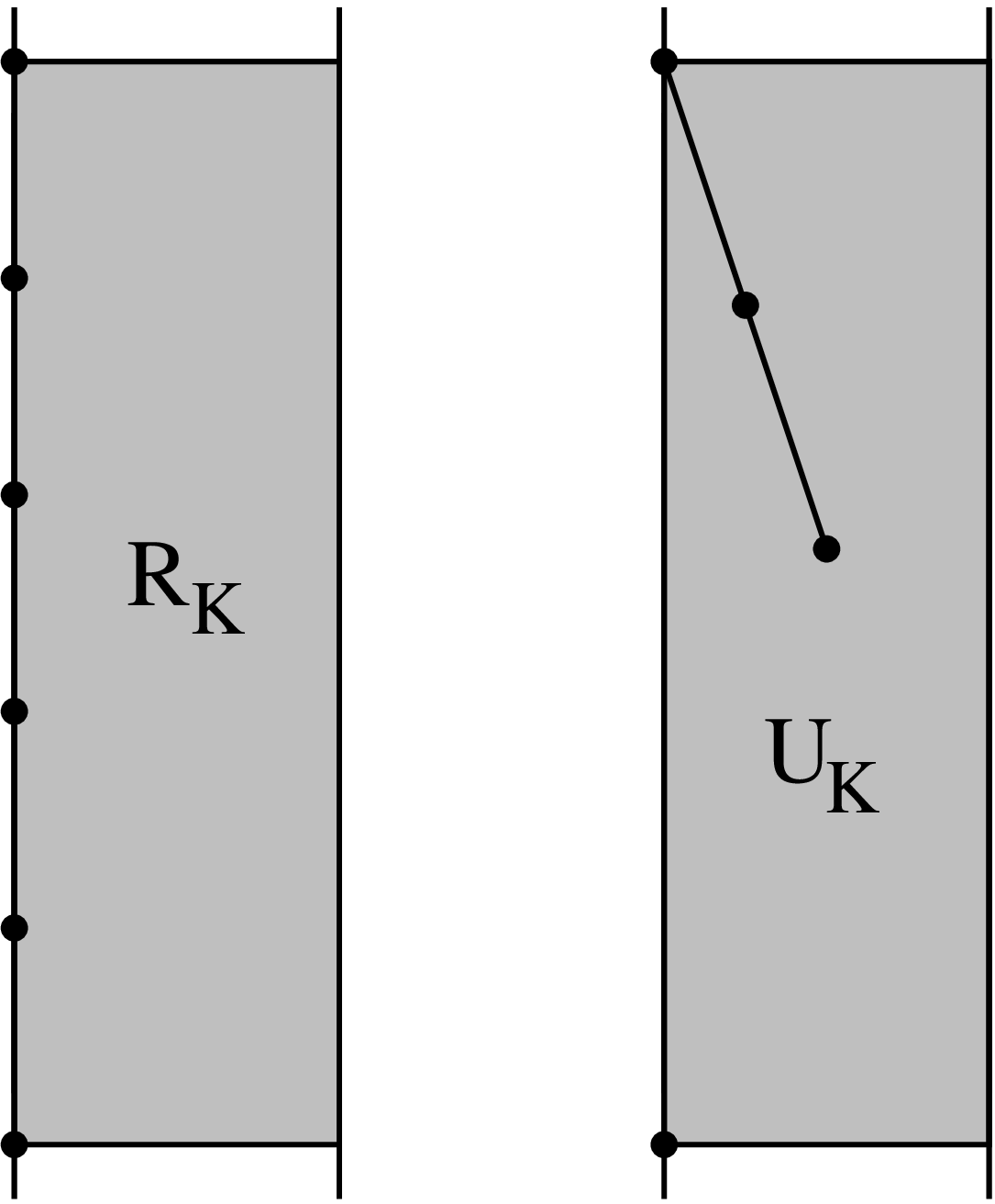}
}
\caption{ \label{Uk}
Definition of $U_K$
}
\end{figure}

\begin{lemma}[Simple folding] \label {simple folding lemma}
There is a quasiconformal map $ \psi_3 : U_K \to R_K$ so that 
($\psi_3$ depends on $j$ and on $K$, but we drop these parameters
from the notation)
\begin{enumerate}
\item $\psi_3$ is the identity on $\partial R_K \setminus L_1$ 
      (i.e., it is the identity on the  the top, bottom and 
       right side of $R_K$), 
\item $\psi_3^{-1}$ extends continuously to the boundary and 
      is linear on each element of ${\cal J}$ lying in $I_K$,
\item $\psi_3$ maps $I_K$ (linearly) to $J_K$,    
\item for each   $z \in I_K$, $\psi_3^{-1} (z) =  \psi_3^{-1} (z^*) \in X_k$ 
         (i.e., $\psi_3$ maps opposite sides of $X_k$ to paired points
          in $I_k$),
\item the quasiconstant of $\psi_3$ depends only on $|I_K|/|J_K|$, 
i.e., on the number of elements in the block associated to $K$.
It is independent of the original model and of  the choice
of $j$ and $K$.
\end{enumerate} 
\end{lemma} 

We call this  ``simple folding''  because it is 
a simple analog of a more complicated folding procedure 
given in \cite{Bishop-classS}. In the lemma above, the 
image domain is a rectangle with a slit removed and 
the quasiconstant of $\psi_3$ is allowed to grow with $n$, the 
number of block elements. This growth is not important 
in this paper because here we only apply the folding 
 construction in 
cases where this number $n$ is uniformly bounded (this 
will occur in our application  because of Lemma \ref{Blaschke lemma 1}).
In \cite{Bishop-classS}, 
the corresponding values may be arbitrarily large but the 
folding construction there must give a map with uniformly
 bounded quasiconstant regardless.
The construction in \cite{Bishop-classS} removes  a 
collection of  finite trees from $R_k$ and does 
so in a way that keeps the quasiconstant of $\psi_3$ bounded 
independent of $n$ (there are also complications involving 
how the construction on adjacent rectangles are merged).

We want to treat the boundary intervals in ${\cal J}_1$ and
${\cal J_2}$ slightly differently. The precise mechanism 
for doing this is:

\begin{lemma}[$\exp$-$\cosh$ interpolation] \label{exp cosh} 
There is a quasiregular map $\sigma_j: S \to D(0, e^2)$ 
so that 
$$
  \sigma_j(z) = 
\begin{cases}
\exp(z),& z \in J \in {\cal J}_1^j,\\
 e \cdot \cosh(z-1),&  z \in J \in {\cal J}_2^j,\\
\exp(z),&   z \in \rhp+2. \\
\end{cases} 
$$
The quasiconstant of $\phi_j$ is uniformly bounded, independent 
of all our choices.
\end{lemma}  

This lemma will be proven in Section \ref{interpolate sec} and 
is completely elementary.

We now have all the individual pieces needed  to construct the 
interpolation $g_j$ between $e^z$ on $L_2$ and $ B \circ \Psi \circ 
\tau_j^{-1}$ on $L_1$.
Let $U_j $ be $S$ minus all the segments $X_K$ where $K \in {\cal L}_j$
as in Lemma \ref{simple folding lemma}.   Define a  quasiconformal map $
\psi: U_j \to S$ by  
$$  \psi =  \psi_1 \circ \psi_2 \circ \psi_3,$$
and let $g_j = \sigma_j \circ \psi$ map $U_j$ into $D(0, e^2)$.
By definition, each $\psi_i$, $i=1,2,3$ is the identity on $L_2$, 
so $g_j(z)  = e^z $ on $L_2$. For any $K \in{\cal L}_j$,
the  map $\psi$ 
sends the boundary segments of $ \partial U_K$ that lie on 
some $X_K$ linearly onto elements of ${\cal J}_2^j$,   so 
boundary points on opposite sides of  $X_K$ get mapped to 
points that are equidistant from $2 \pi i \integers$ and 
$\cosh$ agrees at any two such points. Thus $g_j$ extends 
continuously across each slit $X_K$. Finally, the map $\psi$ 
was designed so that 
 $g_j$ is continuous on $S$ and agrees
with $B  \circ \Psi \circ \tau_j^{-1} $ on $ L_1$. 
Thus $g_j \circ \tau_j $  continuously interpolates 
between $B \circ \Psi$ on $W$ and $F$ on $\Omega(2)$
and  so defines a  quasiregular  $g$ on the whole plane  
with a uniformly bounded constant. 
Thus by the measurable Riemann mapping theorem 
 there is a quasiconformal $\varphi: \complex \to \complex$
so that $f = g \circ \varphi$ is entire. 

The singular values of $f$ are the same as for $g$. On 
$\Omega( 2)$, $g = F = e^\tau$, so $g$  has no critical 
points in this region. In $U_j$, $g= g_j$ is locally $1$-to-$1$, 
so has no critical points there either. Thus the only  critical 
points of $g$ in $\Omega(1)$ are on the slits $X_K$, then these 
are mapped by $g$ onto the  circle of radius $e$ around 
the origin. Thus every critical value of $g$ (and hence $f$)
must lie in $D(0, e)$. 

 If $g$ has a finite asymptotic value 
outside $\overline{D(0, e)}$, then it must be the limit of 
$g$ along some curve $\Gamma$ contained in a single component
of $\Omega$. Then $e^z$ has a  finite limit along 
$\tau(\Gamma) \subset \rhp$; this is impossible, so 
$f$ has no finite asymptotic values outside $\overline{D(0, e)}$.
Thus $S(f) \subset \overline{D(0, e)}$, and so $f \in \classB$.

This proves Theorem \ref{QR} except for the proof of the 
lemmas.

\section{Blaschke partitions}
\label{Blaschke sec}

In this section we prove Lemma \ref{Blaschke lemma 1}. We start
by recalling some basic properties of the Poisson kernel and  
harmonic measure in the unit disk $\disk$.

The Poisson kernel on the
unit circle with respect to the point $a \in \disk$  is given
by the formula
$$ P_a(\theta) = \frac{1-|a|^2}{ |\eit - a^2|} 
= \frac {1-|a|^2}{ 1- 2 |a| \cos( \theta - \phi) +|a|^2},$$
where $a= |a| e^{i \phi}$. This is the same as
$|\sigma'|$ where $\sigma$ is any M{\"o}bius transformation
of the disk to itself that sends $a$ to zero.
If $E \subset \circle$, we  write 
$$ \omega(E, a,\disk) = \frac 1{2\pi} \int_E P_a(\eit) d \theta,$$
and call this the harmonic measure of $E$ with respect 
to $a$. This is the same as the (normalized) Lebesgue measure of 
$\sigma(E) \subset \circle $ where $\sigma: \disk \to \disk$  is
any M{\"o}bius transformation sending $a$ to $0$.
It is also the same as the first hitting distribution on 
$\circle$ of a Brownian motion started at $a$ (although we 
will not use this characterization).

Suppose $I \subset \circle$ is any proper arc, and, as
before,  let 
$\gamma_I$ be the hyperbolic geodesic in $\disk$ with 
the same endpoints as $I$; then $\gamma_I$ is  a 
circular arc in $\disk$ that is perpendicular to 
$\circle $ at its endpoints.  Let $a_I$ denote the 
point of $\gamma_I$ that is closest to the origin.

\begin{lemma}
  $\omega(I, a_I, \disk) = \frac 12$.  
\end{lemma} 

\begin{proof}
Apply a M{\"o}bius transformation of $\disk$ that 
sends $a_I$ to the origin. Then $\gamma_I$ must map
to a diameter of the disk and  $I$ maps to a semi-circle. 
\end{proof}

Given two disjoint arcs $I, J$ in $\circle$, let 
$\gamma_I, \gamma_J$ be the two corresponding 
hyperbolic geodesics and let $a_I^J$ be the point 
on $\gamma_I$ that is closest to $J$ and let 
$a_J^I$ be the point on $\gamma_J$ that is closest
to $I$. 

\begin{lemma} \label{HM symmetric} 
$ \omega(I, a_J^I, \disk) =\omega(J, a_I^J, \disk) $
\end{lemma}  

\begin{proof}
Everything is invariant under M{\"o}bius maps of the
unit disk to itself, so use such a map to send  $I, J$
to antipodal arcs. Then the conclusion is obvious. 
\end{proof} 

\begin{lemma}
If $z, w \in  \disk$ and $I \subset \circle$, then 
$$ \frac {\omega(I, z, \disk)}
          {\omega(I,  w, \disk)} \leq C
$$
where the constant $C$ depends only on the hyperbolic 
distance between $z$ and $w$. 
\end{lemma}

\begin{proof}
Suppose $\sigma(z) = (z-w)/(1- \overline{w}z)$ maps $w $ to $ 0$. 
Then
$u(z) = \omega(I,  \sigma(z), \disk)$ is a positive harmonic
function on $\disk$, so  the lemma is just Harnack's inequality 
applied to $u$.
\end{proof}

Suppose  $I,J, \subset \circle$ are disjoint closed  arcs   and 
$\dist(I,J) \geq \epsilon\max(|I|,|J|)$. Then  we call $I$ and 
$J$ $\epsilon$-separated. This implies the hyperbolic 
geodesics $\gamma_I, \gamma_J$ are separated in the 
hyperbolic metric (with a lower bounded depending only 
on $\epsilon$), but the converse is not true.

\begin{lemma} \label{close lemma} 
If $I, J \subset \circle $  are $\epsilon$-separated,
then the hyperbolic distance between $a_I$ and $a_I^J$ is 
bounded, depending only on $\epsilon$.
\end{lemma} 

\begin{proof}
Assume $I$ is the longer arc and 
consider hyperbolic geodesic   $S$ that connects
$a_I^J$ and $a_J^I$. Then $S$ is perpendicular to $\gamma_I$
at $a_I^J$, so   if $1-|a_I^J| \ll 1-|a_I|$,   $S$ will 
hit the unit circle without hitting $\gamma_j$.
See Figure \ref{CloseFig}.  
\end{proof}

\begin{figure}[htb]
\centerline{
\includegraphics[height=1.5in]{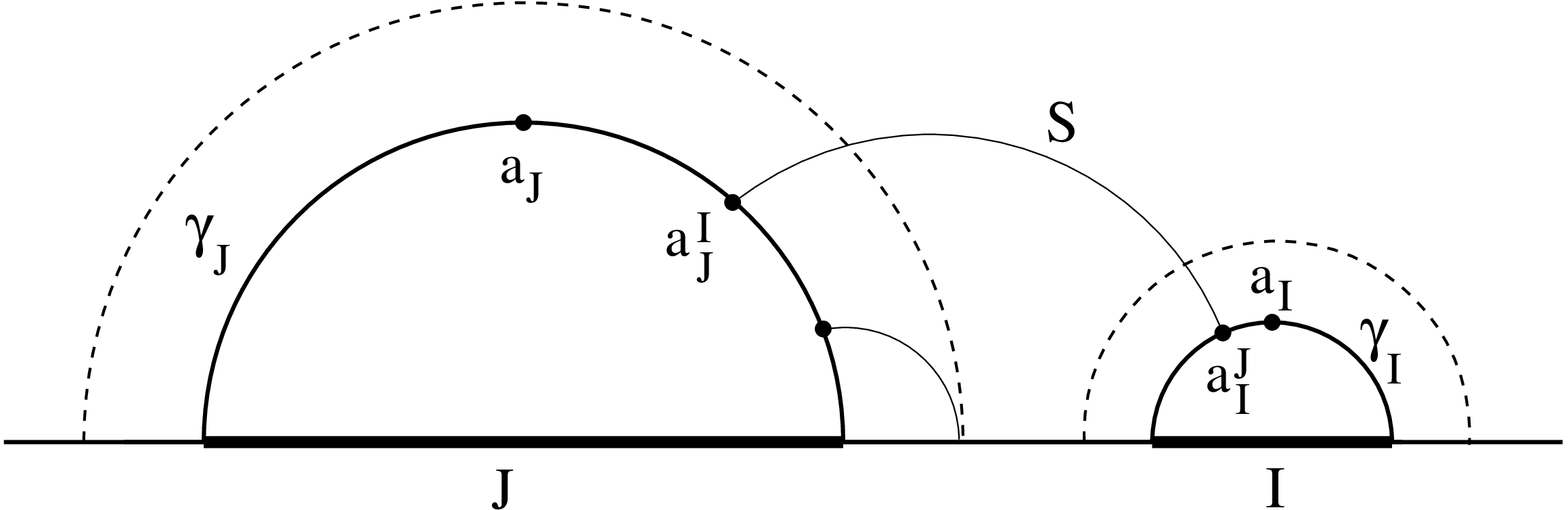}
}
\caption{ \label{CloseFig}
If the intervals $I$ and $J$ are $\epsilon$-separated, 
then a  shortest path between $\gamma_I$ and $\gamma_J$ 
must hit each geodesic near the ``top'' points. A 
perpendicular geodesic that starts too ``low'' 
on $\gamma_J$ will 
hit the unit circle without hitting $\gamma_I$.
}
\end{figure}

\begin{lemma} \label{comparable HM}
Suppose that   $I,J$ are $\epsilon$-separated.
Then 
$$ \omega(I, a_J, \disk)  \simeq \omega(J, a_I, \disk), $$
where the constant depends only on $\epsilon$.
\end{lemma}  

\begin{proof}
This  follows immediately from our earlier results.
\end{proof} 

\begin{lemma}
Suppose that  $I$ and $J$ are $\epsilon$-separated and 
 that $a_J, a_I$ are  at least distance $R$ apart in the 
hyperbolic metric. Then 
$$\omega(J, a_I, \disk) \leq   C(\epsilon) e^{-R}.$$ 
\end{lemma} 

\begin{proof}
Since the intervals are $\epsilon$-separated, the hyperbolic 
distance between $a_I$ and $a_J$ is the same as the 
distance between $a_I^J$ and  $a_J^I$, up to a bounded 
additive factor. Thus if we apply a M{\"o}bius transformation of $\disk$ 
so that $a_J =0$, $a_I$ is mapped to a point $w$ with 
$1-|w| =O(e^{-R})$, which implies $\omega(I, a_J, \disk) = O(e^{-R})$. 
Since the intervals are $\epsilon$-separated, the reverse 
inequality also holds by Lemma \ref{comparable HM}.  
\end{proof}

Fix $M < \infty$ and suppose ${\cal K}$ is a
 collection of disjoint (except
possibly for endpoints) closed  intervals on 
$\circle$ so that any  two adjacent  intervals have 
length  ratio at most $M$. 
We say that two intervals $I,J$ are $S$ steps apart if 
there is a chain of $S+1$ adjacent intervals $J_0, \dots, 
J_S$ so that $I = J_0$ and $J = J_S$.

 Note that if $I, J \in {\cal K}$
 are adjacent, then $a_I, a_J$ are at bounded hyperbolic distance 
$T$ apart (and $T$ depends only on $M$). Also, if $I, J \in {\cal K}$
are not adjacent, then they are $\epsilon$-separated for some 
$\epsilon >0$ that depends only on $M$.

\begin{lemma}
For any $R >0$ there is a collection ${\cal N} \subset {\cal K}$  so that 
\begin{enumerate}
\item for any $I\in {\cal K}$, there is a $J \in {\cal N}$
      with  $\rho(a_J, a_I) \leq R$
\item for any $I, J\in {\cal N}$, $\rho(a_J, a_I) \geq R $.
\end{enumerate}
\end{lemma}

\begin{proof}
Just let ${\cal N}$ correspond to a maximal collection of the points 
$\{ a_K\}$ with the property that any two of them are hyperbolic
distance $\geq R$ apart.
\end{proof} 

Fix a positive integer $S$.  For each $J\in {\cal N}$ choose
the shortest element of ${\cal K}$ that is at most $S$ 
steps away from $J$. Let ${\cal M} \subset {\cal K}$ be the
 corresponding collection of intervals.

\begin{lemma} \label{bound HM sum}
Suppose $R$, $S$, $T$ are as above and $R \geq 4 ST$.
If ${\cal K} $ and ${\cal M}$ are as above,
 then for all $K \in {\cal K}$, 
$$  \epsilon \leq     \sum_{J \in {\cal M}} 
\omega(K, a_J, \disk) \leq \mu,$$
where $\epsilon >0$ depends only on $R$ and $\mu \to  1/2$ as 
$S \to \infty$.  
\end{lemma} 

\begin{proof}
The left-hand inequality is easier and we do it first. 
 Fix  $K \in {\cal K}$.
 There is a $I \in {\cal N}$ with $\rho(a_I, a_K) \leq R$, 
and since adjacent elements of ${\cal K}$ have points that are only 
$T$ apart in the hyperbolic metric, there is an element $J \in {\cal M}$ 
with $\rho(a_K, a_J) \leq R+ST \leq \frac 54 R$.
This implies $|J| \simeq | K| \simeq \dist(J,K)$  and these
imply $\omega(K, a_J, \disk) \geq \epsilon$ with $\epsilon$ depending 
only on $ \rho$.  Thus  every element of ${\cal K}$ has harmonic measure 
bounded below with respect to some point corresponding to 
a  single element of ${\cal M}$ and hence the 
sum of harmonic measures over all elements of ${\cal M}$ is also 
bounded away from zero uniformly.

Now we prove the right-hand inequality.  By our choice 
of $R$, points $a_J$ corresponding to distinct intervals 
in  ${\cal M}$ are at least distance $R/2$ apart. 
Fix $K \in {\cal K}$. There is at most one   point
 within hyperbolic distance $R/4$ of $a_K$ 
and the harmonic measure it assigns $K$
 is at most $1/2$ since the point lies on
 or outside the geodesic $\gamma_K$.

All other points  associated to elements of 
 ${\cal M}$ are Euclidean distance 
$\geq \exp(R/8)|K|$ away from $K$ or are within this distance 
of $K$, and are within  Euclidean distance $\exp(-R/8)|K|$ 
of the unit circle (this is because of the Euclidean 
geometry of hyperbolic balls in the half-space). 
We call these two disjoint sets
${\cal M}_1$ and ${\cal M}_2$ respectively.

Using Lemma \ref{comparable HM} we see that the 
$$  \sum_{J \in {\cal M}_1} \omega(K, a_J, \disk)  
  =O(  \sum_{J \in {\cal M}_1} \omega(J, a_K, \disk))
  = O(\exp(-R/8)).$$  

To bound  the sum over ${\cal M}_2$,  we note that 
each interval in ${\cal M}_2$,  is the endpoint of 
 a chain of $S$ adjacent intervals that are each 
at least as long as $J$.  Since 
$$ |J|\leq \exp(-R/8) |K|,$$
and 
$$ \dist(J,K)  \gtrsim |K|,$$
we can deduce 
$$ \omega(J, a_K, \disk) \leq O(\frac 1S) \omega(a_K, J, \disk),$$
so  since the $J$'s are all disjoint intervals, 
$$  \sum_{J \in {\cal M}_2} \omega(K, a_J, \disk)  
  =O(   \frac 1S \sum_{J \in {\cal M}_2} \omega(J, a_K, \disk))
  = O(\frac 1S).
$$
Choosing first $S$ large, and then $R$ large (depending on $S$ and 
separation constant of ${\cal K}$), both sums are as small 
as we wish, which proves the lemma. 
\end{proof}

\begin{cor} \label{interpolating} 
Suppose $B$ is as above and $K \in {\cal K}$. Then 
$$ \epsilon \leq \frac 1{|K|} \frac{\partial B} {\partial \theta}  \leq C .$$
\end{cor}  

\begin{proof}
If $I,J$ are $\epsilon$-separated, then it is easy to verify that 
$$ \sup_{z \in J} P_{a_I}(z), \qquad \inf_{z \in J} P_{a_I}(z) ,$$
are comparable up to  a bounded multiplicative factor that depends
only on $\epsilon$. The lemma then follows from our earlier estimates.
\end{proof}

We have now essentially proven Lemma \ref{Blaschke lemma 1}; 
it just remains to reinterpret the terminology a little. 
For the reader's convenience we restate the lemma.

\begin{lemma} [The Blaschke partition]
There is a subset ${\cal M} \subset {\cal K}$ so
that if $B$ is the Blaschke product corresponding to
${\cal M}$ and ${\cal L}_j$ is the partition of $L_1$
corresponding to $B$ via $\tau_j \circ \Psi^{-1}$, then
each element of ${\cal J}$ hits at least $2$
elements of ${\cal L}_j$ and at most $M$
 elements of ${\cal L}_j$, where $M$ is uniform.
 In particular, no element of ${\cal J}$ can
hit both endpoints of any element of ${\cal L}_j$ (elements
of each partition are considered
as closed intervals).
\end{lemma}

\begin{proof} 
A computation shows that for the Blaschke product 
$$ B(z) = \prod_n \frac {|a_n|}{a_n} \frac  {  z-a_n}{ 1- \bar a_n z}, $$
the derivative satisfies 
$$ | \frac {\partial B}{\partial \theta} (\eit)| =  \sum _n  P_{a_n}(\eit), $$
and the convergence is absolute and uniform on any 
compact set $K$ disjoint from the singular set
$E$ of $B$ (since $B$ is a product of M{\"o}bius transformations,
and the derivative of a M{\"o}bius transformation is
a Poisson kernel, this formula is simply the limit 
of the $n$-term product formula for derivatives).  

Lemma \ref{bound HM sum} now says  we can choose ${\cal M}$ so that 
$$   2 \pi \epsilon \leq      
 \int_J |  \frac {\partial}{\partial \theta} B | d \theta    
\leq   \frac  34  \cdot 2 \pi  = \frac {3 \pi}2 .$$
Since the integral over an element of ${\cal L}$ has integral exactly 
$2 \pi$, the lower bound  means that an element of ${\cal L}$ 
can contain at most $1/\epsilon$ elements of ${\cal J}$ and hence can 
intersect at most $2+ \frac 1 \epsilon$ elements of ${\cal J}$.
The upper bound says that each element $K$ of ${\cal L}$ must hit at 
least $2$ elements of ${\cal J}$. Hence it is not contained 
in any single element of ${\cal J}$, and so no single 
element of ${\cal J}$ can hit both endpoints of $K$.
\end{proof} 

\section{Straightening a biLipschitz map} \label{linearize sec}

\begin{lemma} \label{linearize2}
Suppose $K =  [1+ia,1+ib] \in  {\cal L}_j$ and define
$$ \alpha(1+iy) = \frac 1{2\pi} \arg( B \circ \Psi \circ \tau_j^{-1}(1+iy)),$$
where we choose a branch of $\alpha$  so $\alpha(1+ia) =0$
(recall that $B(\Psi(\tau_j^{-1}(1+ia))) =1 \in \reals$). Set
$$\psi_1(z) =   1+i(a(1-\alpha(z)) + b\alpha(z)) =1+ i(a + (b-a)\alpha(z)).$$
Then $\psi_1$ is a homeomorphism from $K$ to itself so  that
$ \alpha \circ \psi^{-1}_1 :K \to [0, 1]$ is linear
and $\psi_1$ can be extended to a quasiconformal
homeomorphism of $R= K  \times [1,2]$ to itself
that is the identity on the $\partial R \setminus K$
(i.e., it fixes points on the top, bottom and right side
of $R$).
\end{lemma}

\begin{proof}
The linearizing property of $\psi_1$ is clear from its definition, 
so we need only verify the quasiconformal extensions property.

Corollary \ref{interpolating} implies $\alpha'$ is bounded above and below 
by absolute constants. 
Let $R = K \times [ 1,2]$  and define  an extension of $\psi_1$ 
by 
$$  \psi_1(x+i y) = u(x,y)  +i v(x,y) = 
x  +i[ (2-x) \psi_1(1+iy)+ (x-1) y)].$$
i.e., take the linear interpolation between $\psi_1$ on $L_1$ and 
the identity on $L_2$.   We can easily compute 
$$ 
\begin{pmatrix}
      u_x  & u_y \\ v_x & v_y 
\end{pmatrix} 
= 
\begin{pmatrix}
         1 & 0 \\ y-\psi(y) & (2-x)(b-a)\alpha'(y)+(x-1)
\end{pmatrix} .
$$
Note that $|y-h(y) | \leq|K|$  is absolutely bounded. 
Also, since $|b-a|| \alpha'| $ is bounded above and 
away from $0$, so is  $v_y$.
Thus the derivative matrix lies in a compact subset of the invertible 
$2\times 2$ matrices and hence $\psi_1$ is quasiconformal (with only 
a little more work we could compute an explicit bound for 
the quasiconstant, and even prove that the extension is actually 
biLipschitz).
\end{proof}

\section{Aligning partitions} \label{partition sec}

Now we prove Lemma \ref{partition map}, which we restate for 
convenience.

\begin{lemma} 
There is a $1$-to-$1$, order preserving map of ${\cal L}_j$ into 
(but not necessarily onto) ${\cal J}$ so that
each interval $K \in {\cal L}_j$ is sent to  
an interval $J$ with $\dist(K, J) \leq 2 \pi$.
Moreover, adjacent elements of ${\cal L}_j$ map
to elements of ${\cal J}$ that are either adjacent
or are separated by an even number of elements
of ${\cal J}$.
\end{lemma}

\begin{proof}
For each $K \in {\cal K}$  choose $J \in {\cal J}$ so that
$J$ contains the lower endpoint of $K$ (if two such intervals 
contain the endpoint, choose the upper one). No interval 
$J$ is chosen twice, since Lemma \ref{Blaschke lemma 1} says 
that no $J$ can hit both endpoints of any element of 
${\cal L}$. 

Fix  an order preserving labeling of  the chosen ${\cal J}$ by
$\integers$ and denote it $\{ J_n\}$. 
 By the gap between $J_n$ and $J_{n+1}$ we
mean the number of unselected elements of ${\cal J}$ that 
separate these two intervals. 
The position of $J_0$ is fixed. 
If the gap between $J_0$ and $J_1$ is even 
(including no gap), we leave $J_1$ where it is. If the 
gap is odd, there is a least one separating interval
and we replace $J_1$ by the adjacent interval in ${\cal J}$ 
that is closer to $J_0$.  If the gap between (the new) $J_1$
and $J_2$ is even, we leave $J_2$ alone; otherwise, we move 
it one interval closer to $J_0$.  Continuing in this way, 
we can guarantee that for all $n \geq 0$,   
gaps are even and each $J_n$ is  
either in its original position or adjacent to its 
original position. Thus its distance to the associated
element of ${\cal K}$ is at most $2\pi$. The argument
 for negative indices is identical.
\end{proof} 

\section{Foldings} \label{folding  sec}

Now we prove Lemma \ref{simple folding lemma}. This is the 
step that makes the gluing procedure a little different 
from a standard quasiconformal surgery.

\begin{lemma}[Simple folding] \label {simple folding lemma2}
There is a quasiconformal map $ \psi_3 : U_K \to R_K$ so that
($\psi_3$ depends on $j$ and on $K$, but we drop these parameters
from the notation)
\begin{enumerate}
\item $\psi_3$ is the identity on $\partial R_K \setminus L_1$
      (i.e., it is the identity on the  the top, bottom and
       right side of $R_K$),
\item $\psi_3^{-1}$ extends continuously to the boundary and
      is linear on each element of ${\cal J}$ lying in $I_K$,
\item $\psi_3$ maps $I_K$ (linearly) to $J_K$,
\item for each   $z \in I_K$, $\psi_3^{-1} (z) =  \psi_3^{-1} (z^*) \in X_k$
         (i.e., $\psi_3$ maps opposite sides of $X_k$ to paired points
          in $I_k$),
\item the quasiconstant of $\psi_3$ depends only on $|I_K|/|J_K|$,
i.e., on the number of elements in the block associated to $K$.
It is independent of the original model and of  the choice
of $j$ and $K$.
\end{enumerate}
\end{lemma}

\begin{proof}
The proof is a picture, namely Figure \ref{SimpleFolds}.
The map is defined by giving compatible finite triangulations 
of $R_k$ and $U_k$ (compatible means that there is $1$-to-$1$ 
map between  vertices of the triangulations that preserves 
adjacencies along edges). Such a map defines linear maps between 
corresponding triangles that are continuous across edges. Since 
each such map is non-degenerate, it is quasiconformal and hence 
the piecewise linear map defined between $U_k$ and $R_K$ is 
quasiconformal (with quasiconstant given by the worst quasiconstant 
of the finitely many triangles).  
The other properties are evident.
\end{proof}

\begin{figure}[htb]
\centerline{
\includegraphics[height=3in]{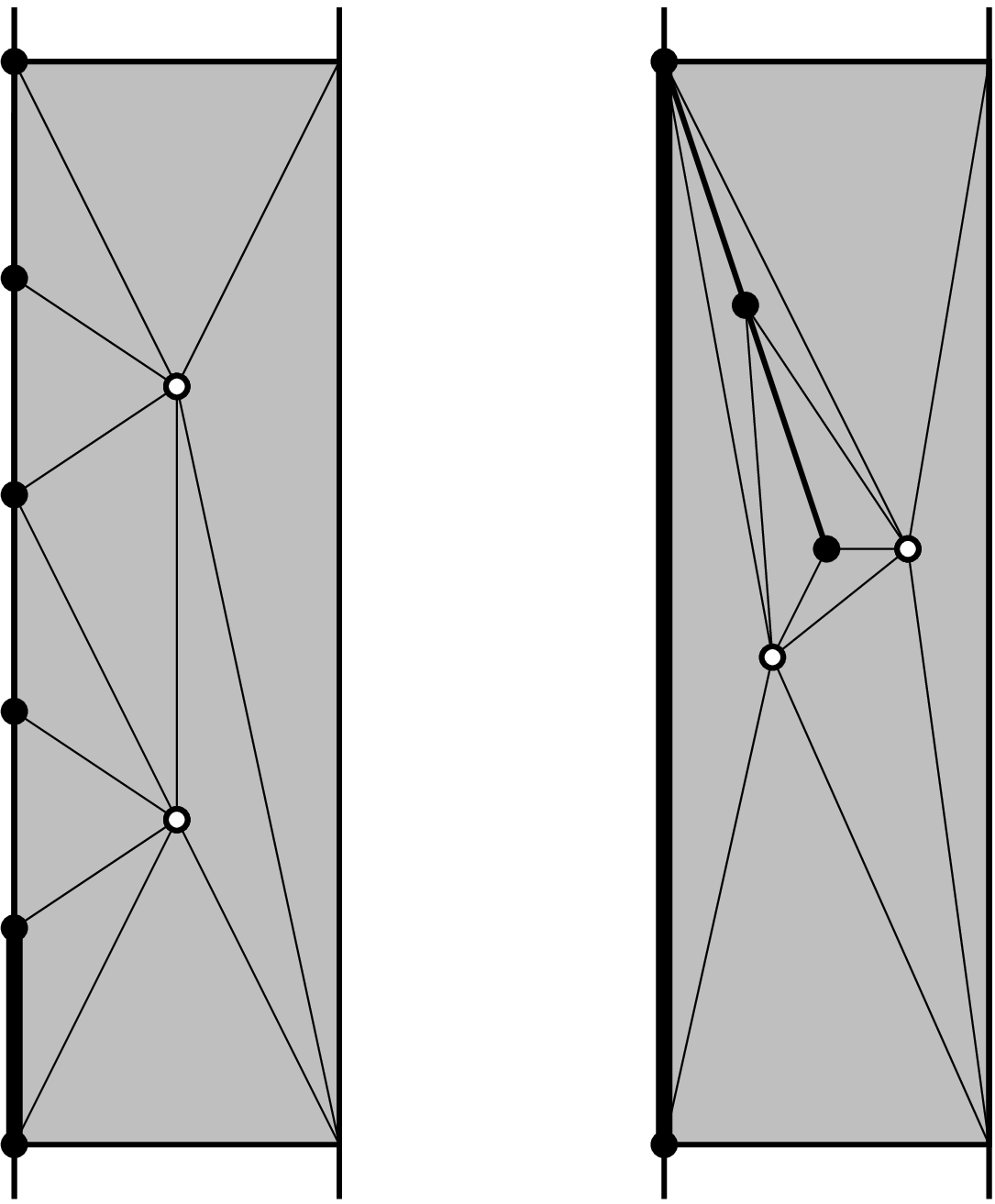}
}
\caption{ \label{SimpleFolds}
The pictorial proof of Lemma \ref{simple folding lemma2} for $n=5$.
}
\end{figure}

\section{Interpolating  between $\exp$ and $\cosh$} \label{interpolate sec}

\begin{lemma}[$\exp$-$\cosh$ interpolation] 
There is a quasiregular map $\sigma_j: S \to D(0, e^2)$
so that
$$
  \sigma_j(z) = 
\begin{cases}
\exp(z),& z \in J \in {\cal J}_1^j,\\
 e \cdot \cosh(z-1),&  z \in J \in {\cal J}_2^j,\\
\exp(z),&   z \in \rhp+2. \\
\end{cases} 
$$
The quasiconstant of $\sigma_j$ is uniformly bounded, independent 
of all our choices.
\end{lemma}

\begin{proof}
As with the previous lemma, the proof is basically a 
picture; see Figure \ref{ExpCosh}. Suppose 
$J \in {\cal J}$ and let $R = J \times [1,2]$. The exponential 
map sends $R$ to the  annulus $A=\{ e < |z| < e^2\}$, with the left 
side  of $R$ mapping to the inner circle and the top and bottom edges 
of $R$ mapping to the real segment  $[e, e^2]$.   

Now define a quasiconformal map $\phi: A \to D(0, e^2)$ that
is the identity on $\{ |z| = e^2\}$ and on  $[e, e^2]$, but
that maps $\{ |z| =e \} $ onto $[-e,e]$  by $z \to \frac 12(z + 
\frac {e^2}z)$ (this is just a rescaled version of 
the Joukowsky map $\frac 12( z + \frac 1z)$ that maps the 
unit circle to $[-1,1]$, identifying complex conjugate points).  

In $\rhp+2$  and in rectangles of the form $J \times [1,2]$ for 
$J \in {\cal J}_1$ we set   $\sigma_j(z) =\exp(z)$. In the rectangles 
corresponding to elements of ${\cal J}_2$ we let 
$\sigma_j(z) = \phi(\exp(z))$. This clearly has the desired 
properties.
\end{proof}

\begin{figure}[htb]
\centerline{
\includegraphics[height=2in]{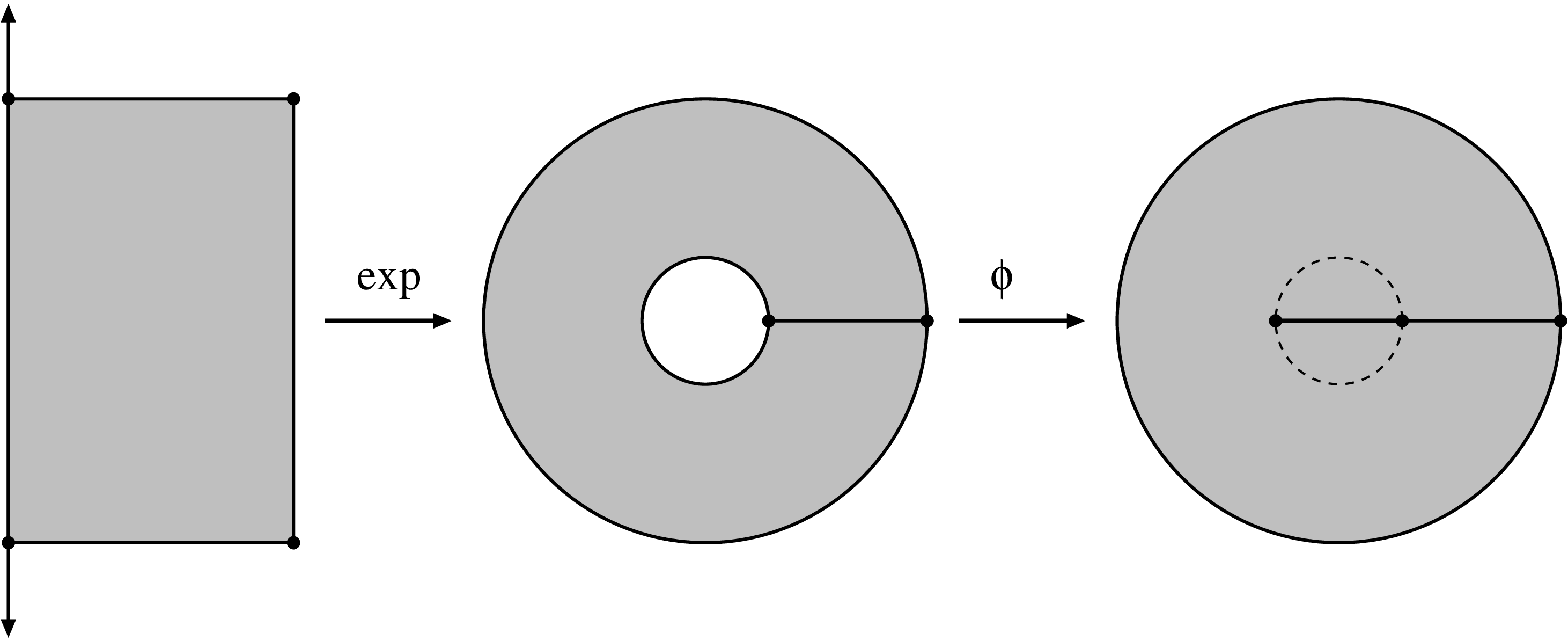}
}
\caption{ \label{ExpCosh}
The exponential function maps the rectangle $[1,2] \times J$
conformally to the slit annulus $\{ e < |z| < e^2\} \setminus 
[e, e^2]$. The map $\phi$ is chosen to map the annulus 
A=$\{ e < |z| < e^2 \}$ 
to the slit disk $\{|z|< e^2 \} \setminus [-e,e]$  so that 
it equals the identity on $\{ |z| = e^2\}$ and equals 
$\frac 12(z + \frac {e^2}{z} )$ on $\{ |z| =e\}$.
}
\end{figure}

Actually, the $\cosh$ function in the lemma can be replaced by any 
function $h:J \to [-1,1]$ that has the property that $ h(z)  $ only 
depends on the distance from $z$ to the endpoint of $J$. This will 
ensure that  after applying a folding map, points that started 
on opposite sides of some slit $X_k$ will end up being identified 
by $h$, which is all we need. 

This completes the proof of Theorem \ref{QR}.

\section{Proof of Theorem \ref{EL models}} \label{Rempe rigidity}

\begin{thm} [Rigidity for disjoint type]
 \label{RempeRigidity disjoint-type}
Suppose $(\Omega, f)$ and $(\Omega',g)$ are disjoint type models, 
$\varphi: \complex \to \complex $ is quasiconformal with $\varphi(  \Omega) 
= \Omega'$ and  $ f = g \circ \varphi$ on $\Omega$.
Then there is a quasiconformal map $\Phi$ of the plane so that
$$ \Phi \circ f = g \circ \Phi,$$
on $\Omega$. In particular $\Julia(g) = \Phi(\Julia(f))$.
\end{thm}

\begin{proof}
The statement and proof are  due to Lasse Rempe-Gillen
\cite{MR2570071}, but we recreate it here for the convenience
of the reader.

Let $W =  \complex \setminus \overline {\Omega}$ and 
$W' =  \complex \setminus \overline {\Omega'}$. We can exhaust 
$W$ by nested open sets  $ U_1
\subset U_2 \subset \dots $ with smooth boundaries and 
$\Omega'$ is exhausted by the images $  \varphi(U_n)$. Since 
the union of these open nested sets covers $\overline{\disk}$ one 
of them covers $\overline{\disk}$, call it $U$.
Thus we can find a new quasiconformal map $\phi:\complex \to 
\complex $ that equals $\varphi$ outside $U$ and is the 
identity on $\overline{\disk}$.

Now inductively 
define a sequence of quasiconformal maps $\{ \Phi_n\}$ on $\complex$\
by setting $\Phi_0$ to be the identity and, in general,
$$ \Phi_{n+1} =
\begin{cases}  g^{-1}  \circ \Phi_n \circ f,& z \in \Omega\\
              \phi, & z \not \in \Omega\\
\end{cases}
$$
Note that since 
 $f: \Omega \to \{ |z|>1\}$ and  $g: \Omega' \to \{ |z|>1\}$ are covering maps, 
the definition of $\Phi_{n+1}$ makes sense as long as  
$\Phi_n$ is a homeomorphism of $\{|z|>1\}$  to itself. 
We shall verify this below.

Set $U_0= U \cap \{ |z| > 1\}$ and let 
  $U_n = \cup_{k=1}^n\{z \in \Omega: f^{k}(z) \in \overline{U}\}$. 
Then $\cup_n U_n$ is the set of all points in $\Omega$
that eventually iterate out of $\Omega$. This is the complement 
of $\Julia(f)$ in $\Omega$ and hence is an open dense set in 
$\Omega$ by Lemma 2.3 of   \cite{MR2570071}.
Let $V_n = \cup_{k=1}^n U_k$.

We claim that 
\begin{enumerate}
\item  for $n \geq 0$, $\Phi_n$ maps  $\{|z|>1\}$ to itself,
\item  for $n \geq 0$, $\Phi_n$ is quasiconformal with the same quasiconstant as $\phi$,
\item  for $n \geq 1$, $\Phi_n = \Phi_{n+1}$  on $V_n$.
\end{enumerate} 
We prove these by induction.  The case $n=0$  for (1) and 
(2)  is trivial since $\Phi_0$ is the identity.
For $n=1$, (3) holds because  if $z \in U_1$ then $f(z) \in U_0$
\begin{eqnarray*}
 \Phi_{2}(z)
& =&  g^{-1}(\Phi_1(f(z)))) 
 =  g^{-1}(\Phi_{0}(f(z)))) 
 =   \Phi_1(z) .
\end{eqnarray*} 
Similarly, for general $n$  (3) holds  because  if $z \in  V_n$, then 
$z \in U_k$ for some $1\leq k \leq n$, so $f(z) \in U_k$ for some 
$0 \leq k \leq n-1$. By the induction hypothesis,     
$f(z) \in f^{-n-1}(U)$, so 
\begin{eqnarray*}
 \Phi_{n+1}(z)
& =&  g^{-1}(\Phi_n(f(z)))) 
 =  g^{-1}(\Phi_{n-1}(f(z)))) 
 =  \Phi_{n}(z). \\
\end{eqnarray*} 

Claim (1) follows from (3) for every $n$  since (3) implies 
$\Phi_n $ is the identity on $U_0$, which contains the unit circle.
Since $\Phi_n$ is a homeomorphism of the plane that means  $\Phi_n$ 
is a homeomorphism of $\{|z|> 1\}$ to itself.

Since $f:\Omega \to \{ |z|>1\}$ and $g: \Omega' \to \{|z|>1\}$
 are holomorphic covering maps, (1) for $n$ implies that 
the first part of the definition of 
$\Phi_{n+1}$ gives a quasiconformal homeomorphism from 
$\Omega$ to $\Omega'$ with the same quasiconstant as 
$\Phi_n$. By induction, this constant is bounded by the quasiconstant 
for $\phi$. Outside $\Omega$, $\Phi_{n+1}$ agrees with $\phi$, 
so again is quasiconformal with constant bounded by 
that of $\phi$. By the Royden gluing lemma (e.g., 
Lemma 2 of \cite{MR0422691}, Lemma I.2 of \cite{MR816367} 
on page 303, \cite{MR0254234} ), 
this implies 
$\Phi_{n+1}$ is $K$-quasiconformal on the whole plane. 
(In many cases of interest, $\partial \Omega$ will be piecewise 
smooth, hence removable for quasiconformal mappings, and then 
the gluing  lemma is not needed.)
Thus all the claims have been established.

Since the  sequence $\{\Phi_n(z) \}$ is eventually constant for 
every $z$ in the dense set $\cup_n V_n \subset
 \Omega_0$, and since $K$-quasiconformal 
maps form a compact family,
 we deduce that $\Phi(z) = \lim_n \Phi_n $ 
defines a $K$-quasiconformal map of the plane. Moreover,
$$ \Phi_{n+1} =
  g^{-1}  \circ \Phi_n \circ f, \quad z \in \Omega
$$
becomes
$$ \Phi =
  g^{-1}  \circ \Phi \circ f, \quad z \in \Omega
$$
in the limit.
\end{proof}

         

\bibliography{models}
\bibliographystyle{plain}

\end{document}